\documentclass[final]{siamart0516}

\usepackage{amsmath,amssymb,enumerate,cleveref,import}

\def\ga{\alpha}
\def\gb{\beta}
\def\gc{\gamma}

\def\th{\theta}

\def\gl{\lambda}

\def\gw{\omega}

\def\bgl{\boldsymbol{\gl}}
\def\bgw{\boldsymbol{\gw}}
\def\bth{\boldsymbol{\th}}

\def\gW{\Omega}

\def\bb{{\bf b}}

\def\bq{{\bf q}}

\def\bs{{\bf s}}
\def\bt{{\bf t}}

\def\bv{{\bf v}}
\def\bw{{\bf w}}
\def\bx{{\bf x}}
\def\by{{\bf y}}
\def\bz{{\bf z}}

\def\b0{{\bf0}}

\def\x{\times}
\def\d{\partial}
\def\<{\langle}
\def\>{\rangle}
\def\bbR{{\mathbb R}}

\def\Lse{\mathfrak{se}}

\def\Ad{\mathrm{Ad\,}}
\def\ad{\mathrm{ad\,}}

\def\({\left(}
\def\){\right)}
\def\rank{\mathrm{rank}\,}

\def\T{\mathrm{T}}

\def\dfrac#1#2{\displaystyle{#1\over#2}}

\usepackage{lipsum}
\usepackage{amsfonts}
\usepackage{graphicx}
\usepackage{epstopdf}
\usepackage{algorithmic}
\ifpdf
  \DeclareGraphicsExtensions{.eps,.pdf,.png,.jpg}
\else
  \DeclareGraphicsExtensions{.eps}
\fi

\newcommand{\TheTitle}{Hyperbolic pseudoinverses for kinematics in the Euclidean group} 
\newcommand{\TheAuthors}{P.~Donelan and J. M.~Selig}

\headers{Hyperbolic pseudoinverses in the Euclidean group}{\TheAuthors}

\title{{\TheTitle}\thanks{Submitted to the editors 2016/08/23.
}}

\author{
  P.~Donelan\thanks{School of Mathematics and Statistics, Victoria University of Wellington, Wellington, New Zealand
    (\email{peter.donelan@vuw.ac.nz}).}
  \and
  J. M.~Selig\thanks{School of Engineering, London South Bank University, London, UK (\email{seligjm@lsbu.ac.uk}).}
}

\usepackage{amsopn}

\ifpdf
\hypersetup{
  pdftitle={\TheTitle},
  pdfauthor={\TheAuthors}
}
\fi

\begin{document}

\maketitle

\begin{abstract}
The kinematics of a robot manipulator are described in terms of the mapping connecting its joint space and the 6-dimensional Euclidean group of motions $SE(3)$.  The associated Jacobian matrices map into its Lie algebra $\Lse(3)$, the space of twists describing infinitesimal motion of a rigid body.  Control methods generally require knowledge of an inverse for the Jacobian. However for an arm with fewer or greater than six actuated joints or at singularities of the kinematic mapping this breaks down.  The Moore--Penrose pseudoinverse has frequently been used as a surrogate but is not invariant under change of coordinates. Since the Euclidean Lie algebra carries a pencil of invariant bilinear forms that are indefinite, a family of alternative hyperbolic pseudoinverses is available.  Generalised Gram matrices and the classification of screw systems are used to determine conditions for their existence.  The existence or otherwise of these pseudoinverses also relates to a classical problem addressed by Sylvester concerning the conditions for a system of lines to be in involution or, equivalently, the corresponding system of generalised forces to be in equilibrium.  
\end{abstract}

\begin{keywords}
  Euclidean group, pseudoinverse, indefinite inner product, manipulator Jacobian, screw system
\end{keywords}

\begin{AMS}
  15A09, 17B45, 70B15
\end{AMS}

\section{Introduction}
\label{s:intro}
The direct or forward kinematics of a robot manipulator is given by a mapping from its joint space $M$, an $m$-dimensional smooth manifold where $m$ is the number of degrees of freedom of its joints, into the Euclidean motion group $SE(3)$ of rigid motions of Euclidean 3-space $E^3$---the space of positions for its end-effector or hand.  Under a choice of orthonormal coordinates for $E^3$, the 6-dimensional Lie group $SE(3)$ is isomorphic to the semi-direct product $SO(3)\ltimes\bbR^3$, where the components represent rotations about the origin and translations, respectively.   Its instantaneous kinematics are given by the derivative of the kinematic mapping at a given configuration $q\in M$, and can be represented by a  Jacobian matrix $J$, giving a linear map from the space $\bbR^m\cong T_qM$ of joint velocities to the space of task or end-effector velocities, i.e. the Lie algebra $\Lse(3)$ associated to the Euclidean group, whose elements are termed {\em twists}.  

A central problem in robotics is to determine the inverse kinematics, enabling a path to be found in the joint space that will give rise to a desired motion of the robot manipulator's end-effector.  The existence of a solution and its uniqueness will depend on the number of joints---the manipulator may have redundancy ($m>6$) or be under-actuated ($m<6$)---and the presence of singularities, that is $\rank J<\min(m,6)$, even in the case $m=6$, is likely to have an effect.  Explicit solutions are known for some classes of serial manipulator, such as wrist-partitioned arms~\cite{pieper}.  However, even when a solution to the inverse kinematics is known it may involve high joint acceleration and torques on the components, so control algorithms are employed.  These generally require inversion of the Jacobian.  If the number of joints is different from six, so the Jacobian matrix is not square, or in the presence of singularities of the kinematic mapping, the need arises to adapt such control algorithms. To that end, Moore--Penrose pseudoinverses have been employed~\cite{whitney}. 

However, there is an intrinsic problem for the Euclidean group with this pseudoinverse~\cite{dot}.  Its definition and existence depend on the choice of positive-definite inner product on the vector spaces involved. In the case of instantaneous kinematics, the matrices in question are Jacobians of the derivative of the kinematic mapping between differentiable manifolds so that one requires a Riemannian metric, conferring such an inner product on each tangent space.  In the case of the joint space $M$, one can assign a positive definite inertia metric.  Given that $SE(3)$ is a Lie group, it is preferable to employ a bi-invariant metric.  However  no such metric exists, as was shown by Loncaric~\cite{loncaric} and Lipkin and Duffy~\cite{lipkin}. The latter pointed out that many robot applications propose using Euclidean metrics on
twists but these metric are not bi-invariant and hence problems can arise. In particular, in some control applications an error signal can depend simply on the choice of origin.  Doty {\em et al}~\cite{dot} observed the related problem of incommensurate units in some robot applications: quantities with different physical units have been combined to produce possible metrics used to compare different systems.  A different choice of units can alter the rank order of systems studied. This problem would be solved if the quantities derived were invariant
with respect to rigid changes of coordinates. Some application may require quantities which are also invariant with respect to change of scale.

Since one would therefore wish the Jacobian matrix to be invariant under (orthogonal) transformation in the coordinates of the ambient space and the end-effector, then the appropriate quadratic form on the Lie algebra should be adjoint invariant.  While there is no bi-invariant positive definite form, there exists a family of bi-invariant indefinite quadratic forms on $\Lse(3)$, the {\em pitch forms} parametrised by the pitch $h\in\mathbb{R}\cup\{\infty\}$ that, apart from the exceptional case $h=\infty$, are non-degenerate.  

The aim of this paper is to determine the extent to which the theory of pseudoinverses can be applied to the instantaneous kinematics of manipulators using the pitch forms.  The theory of pseudoinverses for indefinite inner product spaces has been developed by Kamaraj and Sivakumar~\cite{kamaraj}. There is also a growing literature on the associated question of matrix decompositions (polar and singular value) for indefinite inner product spaces~\cite{boj,bolsh} that also has relevance for robot kinematics, since performance indices such as manipulability~\cite{dot2,paul,yosh} are defined in terms of a positive definite inner product.  The main tools used in the study are Gram matrices~\cite{seldon} and the classification of screw systems due to Hunt and Gibson~\cite{gibson,dongib2}. 

In \cref{s:pinverse}, the theory of hyperbolic $h$-pseudoinverses is presented.  Then, in \cref{s:gram}, the role of the associated Gram matrices  together with the existence of the pseudoinverses is explored.   In particular, the classical problem, solved by Sylvester, of lines in involution is considered from this perspective in \cref{s:involution}.  \Cref{s:none} addresses the existence of screw systems for which no $h$-pseudoinverse exists for any $h$.   Application of these $h$-pseudoinverses to projection operators in shared and hybrid robotic control follows in \cref{s:proj} before some concluding remarks.

\section{Pseudoinverses in the Euclidean Group}
\label{s:pinverse}
\subsection{Invariant forms on the Euclidean Lie algebra}
\label{ss:invariant}
Notation for the Euclidean group and its Lie algebra are briefly presented here.  Further details can be found in~\cite{selig}. Given an orthonormal frame of coordinates $Ox_1x_2x_3$ in Euclidean 3-space $E^3$, the  Euclidean group $SE(3)$ of orientation-preserving rigid displacements is isomorphic to the semi-direct product $SO(3)\ltimes\bbR^3$ of rotations about the origin and translations.  Without such a choice of coordinates, there is no means to identify points of $E^3$ nor elements of $SE(3)$. So, to avoid unnecessary complication, it is usual to assign a preferred choice of coordinates (and hence identify $E^3$ with $\bbR^3$) with respect to which $SE(3)$ will be identified with the semi-direct product. From now on we assume this is the case.

It is frequently useful to identify the translation vector $\bt=(t_1,t_2,t_3)^\T \in\bbR^3$ with the skew-symmetric matrix $T$ whose action on $\bx\in \bbR^3$ is identical to the vector product $T\bx=\bt\x\bx$, namely
\[
T=\begin{pmatrix}0&-t_3&t_2\\t_3&0&-t_1\\-t_2&t_1&0\end{pmatrix}.
\]
Hence an element of $SE(3)$ can be represented by a pair $(R,T)$ where $R\in SO(3)$, so $R^\T  R=I_3$, $\det R=1$, and $T$ is skew-symmetric. Let $R(\th_i)$ denote rotation about the axis $Ox_i$ by an angle $\th_i$ and $T(t_i)$ denote translation by $t_i$ parallel to $Ox_i$, $i=1,2,3$. Composition in the semi-direct product is given by:
\[
(R_2,T_2)\circ(R_1,T_1)=(R_2R_1,R_2T_1R_2^\T +T_2).
\]
Note that the translations form a normal subgroup while the rotation subgroup is not normal. 

The Lie algebra $\Lse(3)$ of the Euclidean group is a 6--dimensional real vector space.  Writing $I_3,O_3$ for the $3\x3$ identity and zero matrices respectively, let
\begin{equation}
\label{eq:sebasis}
\bgw_i=\left.\frac{d}{d\th_i}(R(\th_i),O_3)\right|_{\theta_i=0},\quad \bv_i=\left.\frac{d}{dt_i}(I_3,T(t_i))\right|_{t_i=0},\quad i=1,2,3.
\end{equation}
Then $(\bgw_1,\bgw_2,\bgw_3,\bv_1,\bv_2,\bv_3)$ forms a basis for the Lie algebra.  The first three elements are infinitesimal rotations about the coordinate axes and the last three infinitesimal translations parallel to those axes.  Coordinates with respect to this basis are termed {\em Pl\"ucker coordinates} and they form a natural generalisation of the coordinates of the same name used for lines in projective 3-space. An element $\bs\in\Lse(3)$ is called a {\em twist}.  By a slight abuse of column vector notation, we generally write a twist in Pl\"ucker coordinates  as $\bs=(\bgw,\bv)$.  It is also convenient to introduce the skew-symmetric matrices $\gW$ and $V$ that are determined by the actions: for all $\bx\in\bbR^3$, $\gW\bx=\bgw\x\bx$, $V\bx=\bv\x\bx$, respectively. 

A change of orthonormal frame in $E^3$ (or origin and axes in $\bbR^3$) corresponds to conjugation in $SO(3)\ltimes\bbR^3$ and this gives rise to the adjoint action of the group on its Lie algebra.  Specifically, this action is described in Pl\"ucker coordinates by:
\begin{equation}
\label{eq:Adj}
\Ad(R,T)(\bgw,\bv)=\begin{pmatrix}R&O\\TR&R\end{pmatrix}\begin{pmatrix}\bgw\\\bv\end{pmatrix}.
\end{equation}
Differentiating the adjoint action with respect to $SE(3)$ gives the adjoint of the Lie algebra, namely its Lie bracket. This has the (partitioned) matrix form:
\begin{equation}
\label{eq:adj1}
\ad(\bgw_1,\bv_1)(\bgw_2,\bv_2)=\begin{pmatrix}\gW_1&O\\V_1&\gW_1\end{pmatrix}\begin{pmatrix}\bgw_2\\\bv_2\end{pmatrix}
\end{equation}
or in Pl\"ucker vector form:
\begin{equation}
\label{eq:adj2}
\ad(\bgw_1,\bv_1)(\bgw_2,\bv_2)=[(\bgw_1,\bv_1),(\bgw_2,\bv_2)]=(\bgw_1\x\bgw_2,\bgw_1\x\bv_2+\bv_1\x\bgw_2).
\end{equation}

A natural geometry intrinsic to $\Lse(3)$ is given by a bilinear form $\<\cdot,\cdot\>$ invariant under the adjoint action of the Lie group or algebra in the sense that for all $\bs_1,\bs_2\in\Lse(3)$ and $G\in SE(3)$:
\begin{subequations}
\label{eq:inv}
\begin{align}
\<\bs_1,\bs_2\>&=\<\Ad(G)(\bs_1),\Ad(G)(\bs_2)\>\label{eq:inv1}\\
\intertext{or, equivalently, for all $\bs\in\Lse(3)$}
0&=\<[\bs,\bs_1],\bs_2\>+\<\bs_1,[\bs,\bs_2]\>\label{eq:inv2}
\end{align}
\end{subequations}
The following theorem describing the invariant forms on $\Lse(3)$ is well known. An equivalent version relates to the form of pseudo-riemannian metrics on $SE(3)$~\cite{loncaric}. 

\begin{theorem}
\label{th:invform}
A quadratic form $\<\cdot,\cdot\>$ on $\Lse(3)$ is adjoint-invariant if and only if it has the form, for all $\bs_1,\bs_2\in\Lse(3)$:
\begin{equation}
\label{eq:invform}
\<\bs_1,\bs_2\>_{\ga,\gb}=\frac12\bs_1^\T \begin{pmatrix}-2\ga I_3&\gb I_3\\\gb I_3&O_3\end{pmatrix}\bs_2=:\bs_1^\T  Q_{\ga,\gb}\bs_2,
\end{equation}
for some $\ga,\gb\in\bbR$.
\end{theorem}
Any two such forms for which the pairs $(\ga,\gb)$ differ only by a non-zero multiplicative constant are essentially the same, so that there is a pencil of invariant forms.  If $\gb\neq0$, then let $h=\ga/\gb$ and let $Q_h$ denote the matrix in \cref{eq:invform} scaled by $1/\gb$ and $\<\cdot,\cdot\>_h$ the associated bilinear form.  In the case $\gb=0$ then for $\ga=1$, denote the matrix and form by $Q_\infty$ and $\<\cdot,\cdot\>_\infty$ respectively.  Clearly $Q_{\ga,\gb}=\ga Q_\infty+\gb Q_0$. Note that $Q_\infty$ is one-half of the Killing form of $\Lse(3)$; $Q_0$ is termed the {\em Klein form}.  Not only do these forms span the invariant quadratic forms but they also generate the ring of all polynomial invariants of the adjoint action~\cite{dongib1}.

It is straightforward to confirm that for $h\neq\infty$, $Q_h$ is non-degenerate but indefinite with index $(3,3)$, while $Q_\infty$ is negative semi-definite of rank~3 (correlating to the fact that $SE(3)$ is not semi-simple). In particular, there is no adjoint-invariant inner product on $\Lse(3)$.  

\subsection{Screws and screw systems}
\label{ss:screws}
In Pl\"ucker coordinates, $\bs=(\bgw,\bv)\in\Lse(3)$, the corresponding quadratic forms are:
\begin{equation}
\label{eq:quadforms}
\bs^\T  Q_h\bs=-h\,\bgw.\bgw+\bgw.\bv
\end{equation}
Hence, if $\bgw\neq\b0$,  $\bs$ lies on the nullcone of $Q_h$ if and only if $h=\bgw.\bv/\bgw.\bgw$ while if $\bgw=\b0$ then $\bs$ lies on the nullcone of every $Q_h$, $h\in\bbR\cup\{\infty\}$. Denote by $q_h$ the quadric hypersurface corresponding to $Q_h$ for $h$ finite, and the 3-dimensional subspace $\bgw=\b0$ for $h=\infty$. 

It is often more convenient to work projectively, that is in $\mathbb{P}\Lse(3)$. Let $\widetilde{q}_h$ denote the projective quadric and                then the collection of subsets:
\begin{equation}
\label{e:partition}
\{\widetilde{q}_h-\widetilde{q}_\infty\,:\,h\in\bbR\}\cup\{\widetilde{q}_\infty\}
\end{equation}
partitions $\mathbb{P}\Lse(3)$. Viewed in this projective 5-space, a one-dimensional subspace in $\Lse(3)$, spanned by a non-zero twist, is a point called a {\em screw}.  The projective hyperquadrics have two rulings by (projective) planes: the $\ga$-planes corresponding to the set of screws whose axes pass through a given point in $E^3$ and the $\gb$-planes consisting of those screws whose axes lie in a given plane in $E^3$. Note that $\widetilde{q}_\infty$ is the $\gb$-plane at infinity.

From the point of view of the kinematics of rigid bodies in 3-space, a twist $\bs=(\bgw,\bv)\in q_h$ determines a (Killing) vector field on $E^3$. In the general case $h\neq\infty$, the corresponding exponential is the family of motions that fix a line in $E^3$ and combine rotation about that line with translation along it in the fixed ratio~$h$.  That is, the integral curves are helices of  pitch~$h$ about the axis.   The component  $\bgw\in\bbR^3$ of the twist's Pl\"ucker coordinates represents the direction vector of the axis; $\bv\in\bbR^3$ combines the moment of the axis about the origin with the translation along it.  The case $h=0$ corresponds to pure rotation: in this case the twist satisfies the Klein quadric $\bs^\T  Q_0\bs=0$.  An element  $\bs\in q_\infty$ corresponds to pure translation parallel to $\bv$. 

The motions determined by the one-parameter subgroups, $\exp \th\bs$, $\bs\in q_h$, are precisely the motions of a rigid body constrained to move relative to the ambient space by a one degree-of-freedom joint, where the joint is:
\begin{itemize}
\item 
revolute (R) if $h=0$
\item
helical (H) if $h\neq0,\infty$
\item
prismatic (P) if $h=\infty$.
\end{itemize}

Robot arms and manipulators typically consist of a chain of $m$~rigid links connected in series by such joints though, in practice, helical joints are rarely used. The final link is referred to as the end-effector and its motion in $E^3$ is a function of the joint variables $\th_1,\dots,\th_m$. This determines the kinematic mapping $f:\bbR^m\to SE(3)$, whose derivative describes the instantaneous kinematics of the end-effector.  If one assumes that for a given $\bth\in\bbR^m$, $f(\bth)=e$, the identity in $SE(3)$, then the derivative is a map $\bbR^m\to\Lse(3)$ and can be represented by a Jacobian matrix $J$ with respect to Pl\"ucker coordinates in $\Lse(3)$. The columns of $J$ are the twists $\bs_1,\dots,\bs_m$ corresponding to the current configuration of the arm's joints. 

The columns of $J$ span its image and form a subspace of $\Lse(3)$ whose dimension is the rank of $J$. Such subspaces are referred to as {\em screw systems} or, if the rank is~$k$, then simply $k$-systems.  There is an action of $SE(3)$ on the set of $k$-systems (a Grassmannian manifold over the Lie algebra) induced by the adjoint action.  A classification of $k$-systems was established by Gibson and Hunt~\cite{gibson} and, along similar lines, by Rico and Duffy~\cite{rico}. A refinement and further properties of the Gibson--Hunt classification were established in~\cite{dongib2}. In particular, the invariance with respect to the adjoint action was a central aspect of the last work cited.

Two twists $\bs_1,\bs_2\in\Lse(3)$ are said to be {\em reciprocal} if $\<\bs_1,\bs_2\>_0=\bs_1^\T  Q_0\bs_2=0$. In particular, elements $\bs\in q_0$ are self-reciprocal.  Given a $k$-system $S$, define its reciprocal subspace by: 
\begin{equation}
\label{e:reciprocal}
S^\perp=\{\bs_1\in\Lse(3)\,:\,\text{for all}\; \bs\in S, \<\bs_1,\bs\>_0=0\}.
\end{equation}
Since $Q_0$ is non-degenerate, $S^\perp$ is a $(6-k)$-system. This can be generalised for $h$ finite:  $\bs_1,\bs_2\in\Lse(3)$ are said to be {\em $h$--reciprocal} if $\<\bs_1,\bs_2\>_h=0$ and for a screw system $S$ define the {$h$--reciprocal screw system}:
\begin{equation}
\label{e:hreciprocal}
S^{\perp h}=\{\bs_1\in\Lse(3)\,:\,\text{for all}\; \bs\in S, \<\bs_1,\bs\>_h=0\}.
\end{equation}

\subsection{$h$-pseudoinverses}
\label{ss:hpi}
The Moore--Penrose pseudoinverse of an $n\x m$ matrix $A$  (see, for example,~\cite{bengrev}) that defines a linear transformation between $\bbR^m$ and $\bbR^n$, each equipped with the standard Euclidean inner product, is an $m\x n$ matrix $A^+$ such that:
\begin{enumerate}[{(P}1{)}]
\item\label{P1} $AA^+A=A$
\item\label{P2} $A^+AA^+=A^+$
\item\label{P3} $(AA^+)^\T =AA^+$ 
\item\label{P4} $(A^+A)^\T =A^+A$
\end{enumerate}
For any matrix $A$, there exists a unique pseudoinverse $A^+$~\cite{penrose}. If $m\leq n$ and $\rank A=m$ (so that it is injective as a linear transformation) then 
\begin{subequations}
\begin{equation}
\label{e:piinj}
A^+=(A^\T  A)^{-1}A^\T 
\end{equation} while if $m\geq n$ and $\rank A=n$ ($A$ surjective) then 
\begin{equation}
\label{e:pisurj}
A^+=A^\T (AA^\T )^{-1}.
\end{equation}
\end{subequations}
However, when $A$ does not have maximal rank then one requires full rank factorisations or algorithmic methods to evaluate $A^+$.

The significance of the inner products is that the matrix transpose in (P\ref{P3}) and (P\ref{P4}) corresponds to the {\em adjoint} transformation. That is, if $L:V\to W$ is a linear transformation between real or complex vector spaces $V,W$ equipped with hermitian inner products $\<\cdot,\cdot\>_V$, $\<\cdot,\cdot\>_W$ then the adjoint of $L$ is the linear transformation $L^*:W\to V$ that satisfies, for all $\bv\in V$, $\bw\in W$:
\begin{equation}
\label{eq:adjtrans}
\<\bv,L^*(\bw)\>_V=\<L(\bv),\bw\>_W.
\end{equation}
The adjoint $L^*$ coincides with the dual of $L$ under the identification of the dual spaces $V^*$, $W^*$ with $V,W$ induced by their respective inner products.  

Recall that the Jacobian matrix of a kinematic mapping is a $6\x m$ matrix representing a linear transformation between the space $\bbR^m$ of joint velocities and $\Lse(3)$, the space of twists.  The joint velocity space carries a natural positive-definite inner product---in physical terms the inertia matrix, which we shall assume to be in standard form and so represented by the identity $I_m$.  However the twist space carries a pencil of indefinite forms $Q_h$ which are non-degenerate for $h$ finite. In the case that there are indefinite non-degenerate inner products on the domain and range of a linear transformation, there is still a well-defined adjoint as in \cref{eq:adjtrans}. If the symmetric matrix representations of the inner products on $V,W$ are $M,N$ respectively then the generalised adjoint of a matrix $A$ representing the linear transformation $L:V\to W$ has the form:
\begin{equation}
\label{eq:genadj}
A^{*MN}=M^{-1}A^*N.
\end{equation}
In particular, for a manipulator Jacobian $J$ there is a one-parameter family of generalised adjoints:
\begin{equation}
\label{eq:jacpi}
J^{*h}=J^\T  Q_h.
\end{equation}

Kamaraj and Sivakumar~\cite{kamaraj} develop the theory of pseudoinverses in this generalised setting.  Adapting this to the case at hand, define an {\em $h$--pseudoinverse} of an $m$--variable manipulator Jacobian $J$ to be an $m\x 6$ matrix $J^{+h}$ such that
\begin{enumerate}[{(hP}1{)}]
\item\label{hP1} $JJ^{+h} J=J$
\item\label{hP2} $J^{+h} JJ^{+h}=J^{+h}$
\item\label{hP3} $(JJ^{+h})^{*h}=JJ^{+h}$ 
\item\label{hP4} $(J^{+h} J)^{*h}=J^{+h} J$
\end{enumerate}
The existence of an $h$-pseudoinverse is not guaranteed, though if it exists then it is unique.  A theorem of Kalman~\cite{kalman}, of which the following is a special case, provides a criterion for existence.

\begin{theorem}
\label{th:hpiexist}
The $6\x m$ Jacobian matrix $J$ has an $h$--pseudoinverse $J^{+h}$ if and only if $\rank J=\rank(JJ^{*h})=\rank(J^{*h} J)$.
\end{theorem}

The analogous equalities for the Moore--Penrose pseudoinverse hold automatically. It is straightforward to show that $\rank A=\rank A^\T  A=\rank AA^\T $. Likewise, the imposition of a positive definite inner product on the space of joint velocities and the non-degeneracy of $Q_h$ entails $\rank(JJ^{*h})=\rank(JJ^\T  Q_h)=\rank(JJ^\T )=\rank J$. However the condition $\rank J=\rank(J^{*h} J)=\rank(J^\T  Q_h J)$ can fail because the indefiniteness of $Q_h$ means that there exist non-zero $h$--self-reciprocal twists for any $h$, that is $\bs\in\Lse(3)$, $\bs\neq\b0$ such that $\bs^\T  Q_h\bs=0$. 

This gives rise to the question whether, for any $J$, there exists at least some $h$ such that the $h$--pseudoinverse $J^{+h}$ exists. This is answered in \cref{s:gram}.  Note however that if $\rank J=6$ (and so $m\geq6$) then the conditions of \cref{th:hpiexist} hold and all $h$-pseudoinverses exist; in fact they coincide with the ordinary Moore--Penrose pseudoinverse.  By analogy with~\cref{e:pisurj}, we can write:
\begin{equation}
\label{e:hpi6}
J^{+h}=(J^\T  Q_h)(JJ^\T  Q_h)^{-1}=J^\T (Q_hQ_h^{-1})(JJ^\T )^{-1}=J^\T (JJ^\T )^{-1}.
\end{equation}
This pseudoinverse is used extensively in the kinematics of redundant manipulators ($m>6$). In the next section the more interesting (from the point of view taken in this paper) case $m<6$ is considered.

\section{Gram Matrices, Screw Systems and  Existence of $h$-Pseudoinverses}
\label{s:gram}

In the case that the dimension of the domain (number of joints) $m<6$, the {\em Gram matrices} $J^\T  Q_h J$  are central to determining the existence of an $h$--pseudoinverse. Suppose that the $6\x m$ matrix $J$, whose columns are twists in $\Lse(3)$, has rank~$m$, so that the twists are independent and span an $m$-system in $\Lse(3)$.  Since the forms $Q_h$ are invariants of the adjoint action, the existence of $J^{+h}$ depends only on the equivalence class of this screw system.  Note that in the maximum rank case the $h$--pseudoinverse is determined by the formula analogous to \cref{e:piinj}:
\begin{equation}
\label{e:hpiinj}
J^{+h}=(J^\T  Q_h J)^{-1}J^\T  Q_h.
\end{equation}

The Gibson--Hunt classification~\cite{gibson,dongib2}  of screw systems of a given rank~$m\leq3$ is based on the way the screw system, $S$, meets the pitch hyperquadrics. Types I~and~II are distinguished by whether $S$ does not, or does, lie entirely within a single $q_h$.  A second level of distinction, $A,B,C,\dots$ distinguishes the (projective) dimension of $S\cap q_\infty$ ($A$ denotes empty intersection, and increasing dimension thereafter).  

Further refinement for type~I is provided by the projective type of the pencil of intersections $S\cap q_h$.  For example, \cref{f:3sys} shows a planar section of a 3--system and its intersections with the family of pitch hyperquadrics $q_h$.  Note that there are three values of $h$ giving singular intersections: a real line pair and two singular points which correspond to complex conjugate line pairs. These three values are termed {\em principal pitches} (see \cref{ss:3sys}) and the corresponding projective type determines class I$A_1$. These principal pitches are, in fact, moduli (continuous families of invariants) for the adjoint action of the Euclidean group on screw systems of given dimension. A further class, I$A_2$ arises when two of the principal pitches coincide. Other moduli appear in the various Gibson--Hunt classes.  

For $m=1$, a 1-system is determined by a single non-zero twist. With respect to an appropriate choice of coordinates, this can be chosen to have the form $(1,0,0;h,0,0)$ if the pitch~$h$ is finite, otherwise $(0,0,0;1,0,0)$. Each of these normal forms determines a unique equivalence class under the adjoint action. However, it may be convenient to collect all types where $h\neq\infty$ into one class; equally the case $h=0$ may be treated as a special case.   Mathematical justification for these different choices can be found in \cite{dongib2}. 

For $m=2, 3$, \cref{tab:2systems} and \cref{tab:3systems} list the classifications with normal forms for a set of generating twists for each class of screw system. The reader is referred to the cited works for further details.  Classification in the cases $m=4, 5$ is given in terms of the type of reciprocal 2-~or~1-system, respectively.

\begin{table}[ht]
\caption{Classification of 2--systems.}
\label{tab:2systems}
\begin{center}
\begin{tabular}{|l|l|l|l|}
\hline
Class	& Normal form			& Class	& Normal form\\
\hline\hline
IA   	& $(1,0,0;h_{\ga},0,0)$ & IIA	& $(1,0,0;h,0,0)$ \\
        & $(0,1,0;0,h_{\gb},0)$ &      	& $(0,1,0;0,h,0)$ \\
\hline
IB   	& $(1,0,0;0,0,0)$       & IIB  	& $(1,0,0;h,0,0)$ \\
        & $(0,0,0;1,p,0)$       &       & $(0,0,0;0,1,0)$ \\
\hline
        &                       & IIC   & $(0,0,0;1,0,0)$ \\
        &                       &       & $(0,0,0;0,1,0)$ \\
\hline
\end{tabular}
\end{center}
\end{table}

\begin{table}[ht]
\caption{Classification of 3--systems.}
\label{tab:3systems}
\begin{center}
\begin{tabular}{|l|l|l|l|}
\hline
Class	& Normal form			& Class	& Normal form\\
\hline\hline
IA$_1$  & $(1,0,0;h_{\ga},0,0)$  &  IIA & $(1,0,0;h,0,0)$  \\
        & $(0,1,0;0,h_{\gb},0)$  &      & $(0,1,0;0,h,0)$  \\
        & $(0,0,1;0,0,h_{\gc})$  &      & $(0,0,1;0,0,h)$  \\
\hline
IA$_2$  & $(1,0,0;h_{\ga},0,0)$  &  IIB & $(1,0,0;h,0,0)$ \\
        & $(0,1,0;0,h_{\gb},0)$  &      & $(0,1,0;0,h,0)$ \\
        & $(0,0,1;0,0,h_{\gb})$  &      & $(0,0,0;0,0,1)$ \\
\hline
IB$_0$  & $(1,0,0;h,0,0)$        &  IIC & $(1,0,0;h,0,0)$  \\
        & $(0,1,0;0,h,0)$        &      & $(0,0,0;0,1,0)$ \\
        & $(0,0,0;1,0,p)$        &      & $(0,0,0;0,0,1)$ \\
\hline
IB$_3$  & $(1,0,0;h_{\ga},0,0)$  &  IID & $(0,0,0;1,0,0)$ \\
        & $(0,1,0;0,h_{\gb},0)$  &      & $(0,0,0;0,1,0)$ \\
        & $(0,0,0;0,0,1)$        &      & $(0,0,0;0,0,1)$ \\
\hline
IC      & $(1,0,0;0,0,0)$        & &\\
        & $(0,0,0;0,1,0)$        & &\\
        & $(0,0,0;1,0,p)$        & &\\
\hline
\end{tabular}
\end{center}
\end{table}

By \cref{th:hpiexist}, $J^{+h}$ exists if and only if the rank of the $m\x m$ matrix $J^\T  Q_h J$ is also~$m$.  The following theorem gives a set of equivalent conditions for this to fail.

\renewcommand{\theenumi}{\roman{enumi}}
\renewcommand{\labelenumi}{\rm{(\theenumi)}}

\begin{theorem}
\label{th:nohpi}
Suppose $J$ is a matrix representing a linear transformation $\bbR^m\to\Lse(3)$ where $m<6$ and $\rank J=m$.  Let $S$ denote the $m$-system formed by the image of the transformation. For any finite $h$, the following are equivalent:
\begin{enumerate}
\item
$\rank(J^\T  Q_h J)<m$.
\item
$S\cap S^{\perp h}$ is non-trivial.
\item
$S$ is contained in the tangent space to $q_h$ at some non-zero point.
\end{enumerate}
\end{theorem}

\begin{proof}
By assumption, the columns of $J$, twists ${\bf s}_i$, $i=1,\dots,m$, say, are linearly independent.  

(i)$\Rightarrow$(ii). Given (i), it follows that the kernel of $J^\T  Q_h J$ is non-trivial, so there exists a non-zero vector $\bgl=(\lambda_1,\dots,\lambda_m)^\T $ such that:
\begin{equation}
\label{e:kernel}
(J^\T  Q_hJ)\bgl={\bf 0}.
\end{equation}
Let ${\bf z}=J\bgl=\lambda_1{\bf s}_1+\cdots+\lambda_m{\bf s}_m\in S$: the linear independence of the twists ensures that $\bz\neq\b0$.  Then \cref{e:kernel} can be written as:
\begin{equation}
\label{e:kernel2}
J^\T  Q_h{\bf z}={\bf 0}.
\end{equation}
The components of this vector equation give:
\begin{equation}
\label{e:hsr}
{\bf s}_i^\T  Q_h{\bf z}=\<{\bf s}_i,\bz\>_h=0,\qquad i=1,\ldots,m,
\end{equation}
so, by definition of the reciprocal system, ${\bf z}\in S^{\perp h}$.

(ii)$\Rightarrow$(iii).  Since $q_h$ is the set of twists, $\bs$, satisfying the condition $\<\bs,\bs\>_h=0$, its tangent space at a point $\by\in q_h$ is the hyperplane: 
\begin{equation}
\label{e:pqt}
T_\by q_h=\{\bs\in\Lse(3)\,:\,\<\bs,\by\>_h=0\}.
\end{equation}
Given $\b0\neq\bz\in S\cap S^{\perp h}$, then $\<\bz,\bz\>_h=0$ so that $\bz\in q_h$. Then, for all $\bs\in S$, $\bz\in S^{\perp h}$ implies $\<\bs,\bz\>_h=0$ hence, by \cref{e:pqt}, $\bs\in T_\bz q_h$.

(iii)$\Rightarrow$(i). Suppose that $S\subseteq T_\bz q_h$, $\bz\neq\b0$, so that for all $i=1,\dots,m$, $\<\bs_i,\bz\>_h=\bs_i^\T Q_h\bz=0$. Since $\bz\in S$, we have $\bz=J\bgl$ for some $\bgl\in\bbR^m$, $\bgl\neq\b0$. Hence $J^\T Q_hJ\bgl=\b0$ and so $\rank(J^\T  Q_h J)<m$.
\end{proof}

An obvious consequence is that the existence of an $h$-pseudoinverse for $J$ is determined by the associated screw system.   Condition (ii) immediately gives rise to the following useful result:

\begin{corollary}
\label{th:hrec}
With $J$ and $S$ as in \cref{th:nohpi}, $J$ has an $h$-pseudoinverse if and only if the matrix $J^{\perp h}$, arising from a basis for $S^{\perp h}$ and representing  a transformation $\bbR^{6-m}\to\Lse(3)$,  has an $h$-pseudoinverse. 
\end{corollary}

In the following section we look at an application involving the existence of $h$-pseudoinverses for $h=0$.

\section{Lines in Involution}
\label{s:involution}
In a pair of papers published in 1861, Sylvester~\cite{syl1,syl2} addressed the question of the circumstances under which $m\leq6$ axes of rotation could fail to provide a complete basis for describing an arbitrary rigid motion in space, a condition he termed being in {\em involution}.  Equivalently, a system of forces acting on these lines would be in equilibrium, a problem initially considered by M\"obius. Sylvester showed that the requisite condition is equivalent to $\det(J^\T Q_0J)=0$, where the columns of the $6\times m$ matrix $J$ are the Pl\"ucker coordinates of the lines in space.  It follows straight away that $m$~linearly independent  lines are in involution if and only if the pseudoinverse $J^{+0}$ does not exist.  Example~1 in \cref{s:none} illustrates this in the case $m=3$ since the columns, as zero-pitch screws, can also be thought of as lines in space. 

Sylvester established geometric and algebraic conditions for involution for various~$m$.  Here we give an alternative characterisation using \cref{th:nohpi} and clarify its connection with the classification of screw systems~\cite{gibson}.  Assume that the $m$~lines are linearly independent.  As noted following the proof of \cref{th:nohpi} the involution condition depends only on the screw system spanned by the lines, not the particular choice of lines in it. Indeed, the Sylvester determinant is invariant with respect to the adjoint action on $\Lse(3)$ and covariant with respect to choice of basis for the screw system. The existence of lines, that is zero-pitch screws, within the screw system places constraints on the possible systems that can arise.  Specifically, working projectively in $\Lse(3)$, the intersection of the screw system as a projective subspace with the (open) pitch quadric $\widetilde{q}_0-\widetilde{q}_\infty$ must contain at least~$m$ distinct points.

For $m=2$, every twist in a 2--system of type II$C$ has infinite pitch and cannot represent a line, so two lines can only be in involution if they span a 2--system of type $A$ or $B$.   Suppose two lines (necessarily in a type $A$ or $B$ system) are $\bs_1$ and $\bs_2$,  so that $\bs_1^\T Q_0\bs_1=\bs_2^\T Q_0\bs_2=0$ since they correspond to pitch zero twists.  The Sylvester condition then reduces to $\det(J^\T Q_0J)=-(\bs_1^\T Q_0\bs_2)^2=0$.  That is, for a pair of lines to be in involution they must be reciprocal, see \cref{ss:screws}.    It is well known that reciprocal lines must be coplanar~\cite{gibson}.  There are two possibilities to consider. First, if the two lines meet, that is to say if they are concurrent, then linear combinations of the two lines generate all lines through the common point of $\bs_1$ and $\bs_2$ and lying in the same plane as the original lines.  Such an arrangement of lines is sometimes referred to as a plane pencil of lines.  The 2--system is precisely this plane pencil of lines and projectively the screw system lies entirely in the open pitch quadric  $\widetilde{q}_0-\widetilde{q}_\infty$, hence is a II$A$ system with modulus $h=0$ (see~\Cref{tab:2systems}).

In the second case, $\bs_1$ and $\bs_2$ are parallel.  In this case the linear combination of a pair of parallel lines gives all the lines coplanar and parallel to the original lines, together with a line `at infinity' in a direction normal to the plane of the finite lines.  That is, the screw system meets $\widetilde{q}_\infty$ in a single point and hence  is a II$B$ system  with modulus $h=0$.


When $m=3$, the types of screw systems that contain three independent lines are restricted to the following (see~\cite{cocgibdon}): I$A$ ($h_{\max}>0$, $h_{\min}<0$), I$B_0$, I$B_3$ (principal pitches must be of opposite sign), II$A$, II$B$, II$C$ (all with principal pitch zero).  In the case I$A$ there are three principal pitches, the middle value corresponding to the quadric intersection with the corresponding pitch hyperquadric being a degenerate line pair  (see \cref{f:3sys}).   With three lines, $\bs_1,\,\bs_2$ and $\bs_3$ the Sylvester determinant reduces to $\det(J^\T Q_0J)=2(\bs_1^\T Q_0\bs_2)(\bs_2^\T Q_0\bs_3)(\bs_3^\T Q_0\bs_1)$.  So the vanishing of this determinant implies that for three lines to be in involution, at least one pair of lines must be coplanar.

Starting with this pair of lines, there are  various possibilities to consider.  First, in the case where the two coplanar lines are concurrent,  the   lines generate a plane pencil of concurrent lines in the screw system.  If the third line meets, but is not contained in the plane of the other two lines then there will be a line in the pencil meeting it that, together with the third line, will generate another plane pencil.  The 3--system will thus contain  two plane pencils of lines in two different planes, thus a I$A$ system with one of the moduli equal to 0 (in this case the moduli are the principal pitches of the screw system).  

Now it may happen that all three lines are pairwise concurrent, that is, the third line lies in the plane determined by the first two.  In this case linear combinations of the three lines will generate any line lying in the plane and hence this is a II$B$ system with modulus $h=0$, corresponding to a $\gb$-plane in the Klein quadric  $\widetilde{q}_0$.  On the other hand, if the lines all meet at a single point but are not all coplanar, then linear combinations of the lines will produce all lines through the common point.  This is sometimes called a bundle of lines and corresponds to an $\ga$-plane in the Klein quadric.  As a screw system it has type II$A$ with modulus $h=0$.

Next, suppose that the third line is parallel to the plane determined by the first two and so does not meet it at all. In the pencil determined by the first two lines there will be a line which is parallel to the third line; this line and the third one will generate a set of parallel lines in the plane determined by the two parallel lines.  As before the set of parallel lines will also contain a twist in $\widetilde{q}_\infty$ and it can be shown that these lines generate a  I$B_0$ 3--system with one modulus $h_a=0$.

Secondly, we consider the cases where two of the original lines are parallel.  All the cases where the third line meets the plane determined by the parallel lines have already been considered. So the only cases to consider are either where the third line lies in a plane parallel to the plane determined by the other two---this gives a  I$B_3$ system---or, finally, where all three lines are parallel, giving a II$C$ $(h=0)$ system.
 
In summary there are six types of 3--system which contain three lines in involution.  They are the I$A$ with one principal pitch zero, the I$B_0$ system with $h_a=0$, the I$B_3$ system, the II$A$, II$B$ and II$C$ systems all with modulus $h=0$.



For $m=4$, the reciprocity condition, \cref{th:hrec}, together with the result for $m=2$ ensure that only four lines spanning a 4--system reciprocal to a 2--system of type II$A$ or II$B$, with principal pitch zero, satisfy the Sylvester condition.  Recall that the lines in a  II$A$ 2--system with $h=0$ form a plane pencil.  So the  reciprocal system will consist of all the lines in the plane of the pencil together with all the lines through the central point of the pencil.  The lines in a 2--system of type II$B$ ($h=0$) are a set of parallel lines in a plane.  The reciprocal system will thus consist of all  lines in the plane together with all lines in space parallel to the lines in the 2--system.  

Hence a set of four linearly  independent lines in involution will comprise a set of four lines from one of these two arrangements.  It follows that the four lines are in involution precisely when they satisfy the special condition that there is an infinity of transversals---a {\em transversal} being a line reciprocal to, and therefore intersecting, all four lines.  This contrasts with four lines in a system reciprocal to a type I$A$ 2--system having two transversals and those in a system reciprocal to a I$B$ system having a single transversal.  A different approach, using Grassmann--Cayley algebra, is presented in Sturmfels, Section~3.4~\cite{sturmfels}.


For five lines ($m=5$) to be in involution the lines must span a 5--system reciprocal to a single line, as opposed to a twist of non-zero pitch, and therefore they possess a common transversal.   Another way to express this is to say that the 5 lines lie in a special linear line complex.  Clearly six lines can only be in involution if they are not independent and so span a 5--system.  

The significance of involution for a robot arm with revolute joints is this.  Suppose the arm is in a non-singular configuration, so that the corresponding locations of the joint screws are given by $m$ independent twists of pitch zero. If the twists correspond to $m$ lines in involution, then there exists a combination of wrenches generated by joint torques with respect to which the end-effector of the arm is in equilibrium.  In fact, the wrench can be expressed as $W=Q_0\bz$ where $\bz$ is the twist identified in the proof of \Cref{th:nohpi}. Note that from the second part of the proof, $\bz\in q_0$ so that it corresponds to a line and must be a transversal to the joint axes.

\section{Systems with No $h$-Pseudoinverse}
\label{s:none}
The question  also arises as to whether there are screw systems for which, for an associated Jacobian matrix, no $h$-pseudo- inverse exists for any $h$.   We consider this dimension by dimension.

An easy consequence of \cref{th:nohpi} is that if $\rank(J^\T  Q_h J)<m$ for all $h$, so that no $h$-pseudoinverse exists, then for each $h$ there exists $\bz\in S$ such that $\<\bz,\bz\>_h=0$.  If $\bz$ has finite pitch $h$ for all $h\neq\infty$, then the screw system spanned by the columns of~$J$ will contain a screw of every pitch and therefore must  include a screw of infinite pitch.  Therefore, in any case, the screw system must contain a screw of infinite pitch.  It follows that all type $A$ systems (which contain no screws of infinite pitch) have an $h$-pseudoinverse for some $h$.  To put it the other way, screw systems that have no $h$-pseudoinverse can only be of types $B$, $C$ or $D$.

\subsection{$1$--systems}
\label{ss:1sys}
Suppose $\b0\neq\bs\in\Lse(3)$ and let $S$ be the 1-system spanned by $\bs$. The $h$-pseudoinverse of $\bs$ exists if and only if $\bs^\T Q_h\bs\neq0$. This fails if and only if $\bs\in q_h$ so for screws of finite pitch almost all $h$-pseudoinverses exist. For $\bs\in q_\infty$, we have $\bs^TQ_h\bs=0$ for all finite $h$ so no $h$-pseudoinverse exists.

\subsection{$2$--systems}
\label{ss:2sys}
For a general 2-system $S$, spanned by $\bs_1,\bs_2$, columns of the Jacobian matrix $J$, the condition for singularity of the Gram matrix:
\begin{equation}
\label{e:gram}
\det(J^\T Q_hJ)=0
\end{equation}
is quadratic in $h$ and there are typically two {\em principal pitches}, say $h_\ga,h_\gb$, for which the screw system is tangent to the pitch hyperquadric~\cite{gibson}.  The coefficients $i_1,i_2,i_3$ of this quadratic $i_1h^2+i_2h+i_3$ are invariants of the screw system~\cite{selig} and \textbf{no} $h$-pseudoinverse exists if and only if all these invariants vanish.   This occurs only for systems of types II$B$ and II$C$, being planes in $\Lse(3)$ that are tangent to every $q_h$ at an infinite pitch screw $\bs_\infty\in S$. 

   \begin{figure}[h!btp]
      \centering
      \includegraphics[scale=0.3]{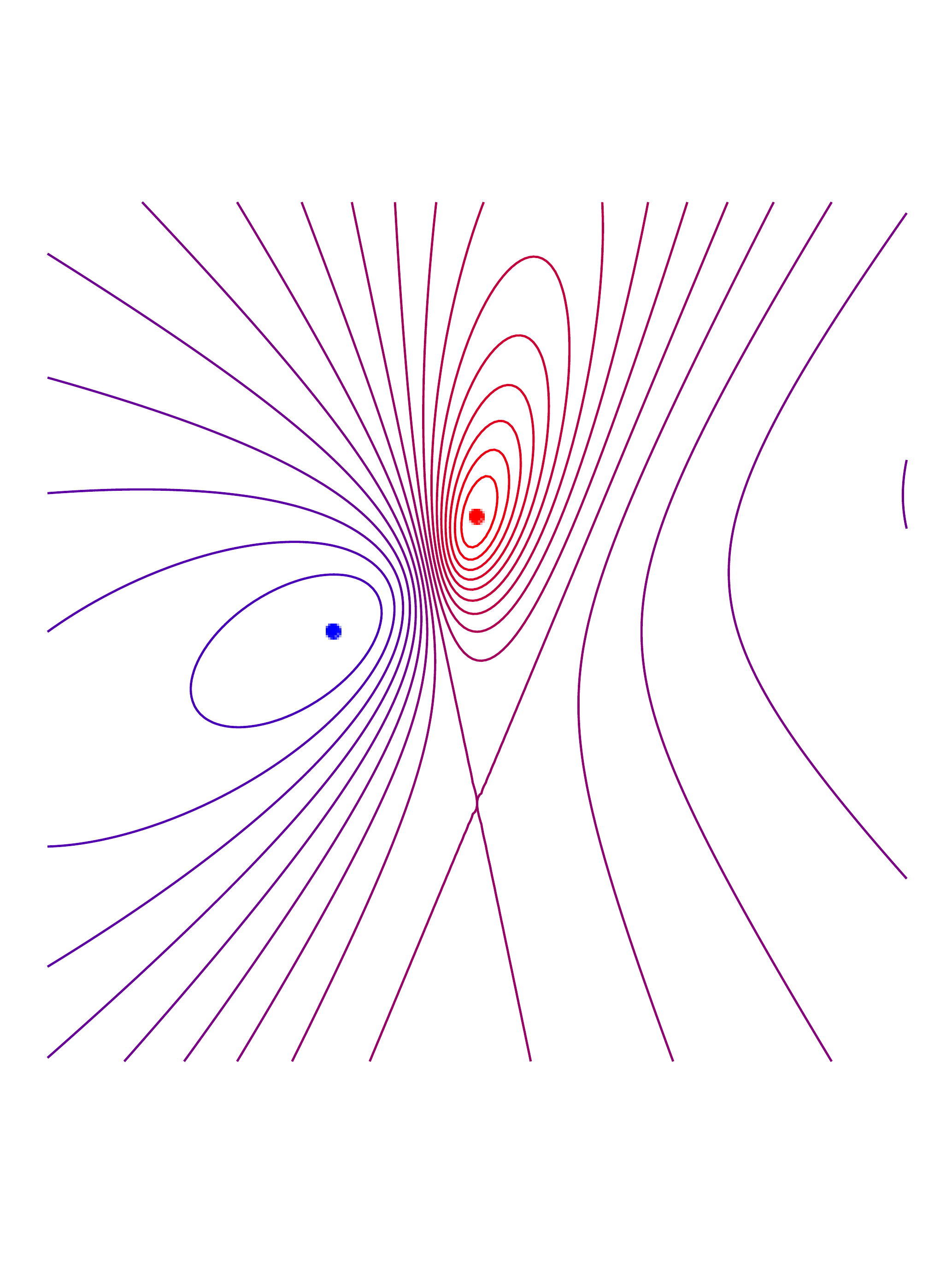}
      \caption{Intersection of type I$A$ 3--system with pitch hyperquadrics}
      \label{f:3sys}
   \end{figure}

\subsection{$3$--systems}
\label{ss:3sys}
Analogously to 2-systems and equation \cref{e:gram}, a typical (type I$A$) 3-system has three principal pitches $h_\ga,h_\gb,h_\gc$ for which the corresponding $h$-pseudoinverse fails. They are the singular quadrics in \cref{f:3sys}.  The determinant vanishes identically in $h$ if and only if the four invariant coefficients vanish.  Selig~\cite{selig} shows that this occurs for the type~II systems that contain a screw of infinite pitch (II$B$, II$C$ and II$D$) but also for two classes of type~I system, namely I$B_3$ and I$C$, corresponding to singular pencils of quadrics~\cite{dongib1}.  Other types will have one or two singular pitches with no corresponding pseudoinverse. 

\subsection{$4$-- and $5$--systems}
\label{ss:4sys}
By \cref{th:hrec}, the existence or otherwise of $h$-pseudo- inverses follows from that of the reciprocal $2$-- and $1$--systems. For example, whereas a generic $5$--system intersects $q_\infty$ in a plane, a $5$--system fails to have any $h$-pseudoinverse if and only if it wholly contains~$q_\infty$. 

A summary of the 10 types of screw systems with no $h$-pseudoinverse is given in \cref{tab:nohpi}. Note that the reciprocals of the four 3-systems have the same type.

\begin{table}[tbhp]
\caption{Screw systems with no $h$-pseudoinverse.}
\label{tab:nohpi}
\begin{center}
\begin{tabular}{|c|l|c|c|c|}
\hline
Dimension	&GH type	& Basis 	& Reciprocal	&Reciprocal
	\\
&	&	 & Dimension & Basis  \\
\hline\hline
1	& $h=\infty$  & $(0,0,0;1,0,0)$  & 5 & $(0,1,0;0,0,0)$  \\
&&&&\\[-2ex]
        &   &   &   & $(0,0,1;0,0,0)$  \\
&&&&\\[-2ex]
        &   &   &   & $(0,0,0;1,0,0)$  \\
&&&&\\[-2ex]
        &   &   &   & $(0,0,0;0,1,0)$  \\
&&&&\\[-2ex]
        &   &   &   & $(0,0,0;0,0,1)$  \\
\hline
2	& II$B$  & $(1,0,0;p,0,0)$  &  4  & $(1,0,0;-p,0,0)$  \\
&&&&\\[-2ex]
        &	&  $(0,0,0;0,1,0)$ &  &  $(0,1,0;0,0,0)$ \\
&&&&\\[-2ex]
        &	&   &  &  $(0,0,1;0,0,0)$ \\
&&&&\\[-2ex]
        &	&   &  &  $(0,0,0;0,0,1)$ \\
\hline
2	& II$C$  & $(0,0,0;1,0,0)$  &  4  & $(0,0,1; 0,0,0)$  \\
&&&&\\[-2ex]
        &	&  $(0,0,0;0,1,0)$ &  &  $(0,0,0;1,0,0)$ \\
&&&&\\[-2ex]
        &	&   &  &  $(0,0,0;0,1,0)$ \\
&&&&\\[-2ex]
        &	&   &  &  $(0,0,0;0,0,1)$ \\
\hline
3	& I$B_3$  & $(1,0,0;p,0,0)$  &  3  & $(1,0,0;-p,0,0)$  \\
&&&&\\[-2ex]
        &	&  $(0,1,0;0,q,0)$ &  &  $(0,1,0;0,-q,0)$ \\
&&&&\\[-2ex]
        &	&  $(0,0,0;0,0,1)$ &  &  $(0,0,0;0,0,1)$ \\
\hline
3	& I$C$  & $(1,0,0;0,0,0)$  &  3  & $(-p,0,1;0,0,0)$  \\
&&&&\\[-2ex]
        &	&  $(0,0,0;0,1,0)$ &  &  $(0,0,0;0,10)$ \\
&&&&\\[-2ex]
        &	&  $(0,0,0;1,0,p)$ &  &  $(0,0,0;0,0,1)$ \\
\hline
3	& II$B$  & $(1,0,0;p,0,0)$  &  3  & $(1,0,0;-p,0,0)$  \\
&&&&\\[-2ex]
        &	&  $(0,1,0;0,p,0)$ &  &  $(0,1,0;0,-p,0)$ \\
&&&&\\[-2ex]
        &	&  $(0,0,0;0,0,1)$ &  &  $(0,0,0;0,0,1)$ \\
\hline
3	& II$C$  & $(1,0,0;p,0,0)$  &  3  & $(1,0,0;-p,0,0)$  \\
&&&&\\[-2ex]
        &	&  $(0,0,0;0,1,0)$ &  &  $(0,0,0;0,1,0)$ \\
&&&&\\[-2ex]
        &	&  $(0,0,0;0,0,1)$ &  &  $(0,0,0;0,0,1)$ \\
\hline
\end{tabular}
\end{center}
\end{table}

\medskip
{\bf Example~1.} Consider the matrix 
\[
J=\begin{pmatrix} 1&0&0\\0&0&0\\0&1&1\\0&0&-2\\-1&0&1\\\frac12&0&0\end{pmatrix}
\]
Its columns are the linearly independent pitch-zero screws $\bs_1$, $\bs_2$, $\bs_3$.  Since $\bs_2-\bs_3\in q_\infty$, this is a I$B_0$ system. We have:
\[
\det(J^\T Q_h J)=8h
\]
so that its principal pitch is $h=0$. Moreover, it can be seen that the rank condition in \cref{th:hpiexist} fails for $h=0$ so this matrix does not have a $0$-pseudoinverse.    It is straightforward to compute:
\[
J^{+h}=\begin{pmatrix}
1&-\dfrac12&0&0&0&0\\[2ex]
\dfrac3{16h}&-\dfrac{16 h^2+16 h+3}{32h}&1&\dfrac12&0&-\dfrac3{8 h}\\[2ex]
\dfrac1{16h}&\dfrac{16 h^2+16 h+3}{32h} &0&-\dfrac12&0&-\dfrac1{8h}
\end{pmatrix}
\]

\section{Projection Operators and Robotics}
\label{s:proj}
\subsection{Manipulator Kinematics}
In this section we return to the problem mentioned in the introduction, of controlling a robot with less than 6 degrees-of-freedom.   This problem is important for robotics but is also of interest in computer animation, see \cite{Buss+Kim} for example.  As an example,  consider the following robot whose kinematic structure is typical of many commercially available robots for hobbyists and for educational purposes.

   \begin{figure}[h!btp]
      \centering
      \def\svgwidth{7cm}
      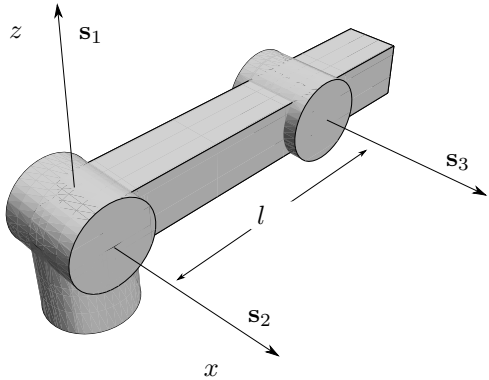
      \caption{A 3R Robot Arm}
      \label{f:3rob}
   \end{figure}

\bigskip
\noindent{\bf Example~2.}   Consider the 3-joint robot arm illustrated in Figure \ref{f:3rob}.  This arm has 3 revolute (R) or hinge joints.  The joints can be represented as pitch 0 twists or lines as noted above.  With  coordinates as shown in the figure, the Jacobian matrix of the robot will be:
\[
J=\begin{pmatrix} 0&1&1\\ 0&0&0\\ 1&0&0\\  0&0&0\\   0&0&0\\ 0&0&-l\end{pmatrix}
\]
Now suppose that, in its illustrated configuration, it is desired that the end-effector of the robot move with a  velocity given by the twist $\bs_d$.  If the joint velocities of the three revolute joints are $\dot{\theta}_1,\,\dot{\theta}_2$ and $\dot{\theta}_3$, then the velocity of the robot's end-effector will be  $J\dot{\bth}$, where $\dot{\bth} =(\dot{\theta}_1,\,\dot{\theta}_2,\,\dot{\theta}_3)^\T$.  Since the Jacobian has maximum rank~3, clearly not every desired twist can be achieved by the robot, but we would like to find values for the joint rates in such a way that the robot's (twist) velocity approximates the desired twist and is exactly the desired twist, $\bs_d=J \dot{\bth}$, when $\bs_d$ does lie in the screw system spanned by the joints of the robot.  In effect, we seek a projection operator which projects the whole Lie algebra onto the column space of the Jacobian.  In view of the property (hP1) of  $h$-pseudoinverses, it is clear that $P_h=JJ^{+h}$ satisfies this requirement.

For this example, there is no $0$-pseudoinverse, since two of the axes are coplanar; in fact, joints 1 and~2 are concurrent and joints 2 and~3 are parallel, so these lines generate a I$B_0$ 3--system (see Section~\ref{s:involution} above).  However, for any other finite $h$ we have,
\[
P_h=JJ^{+h}=
\begin{pmatrix}
1&0&0&-\frac{1}{2h}&0&0\\
0&0&0&0&0&0\\
0&0&1&0
&0&0\\
0&0&0&0&0&0\\
0&0&0&0&0&0\\
0&0&0&0&0&1
\end{pmatrix}
\]

Notice that the kernel of this operator depends on the value of $h$; in fact the twists:
\[
\bz_1=\begin{pmatrix}0\\ 1\\ 0\\0\\ 0\\0\end{pmatrix},\,
\bz_2=\begin{pmatrix}0\\ 0\\0\\ 0\\1\\0\end{pmatrix}\quad\mbox{and}\quad
\bz_3=\begin{pmatrix}1\\ 0\\ 0\\ 2h\\0\\0\end{pmatrix},
\]
span the kernel of $P_h$.  

More generally, suppose that $S$ is the linear system of twists generated by the columns of some Jacobian matrix $J$, then it is easy to see that the kernel of the projection operator $P_h=JJ^{+h}$ is just the space of $h$-reciprocal twists, $S^{\perp h}$.  Since  $P_h=J(J^\T Q_hJ)^{-1}J^\T Q_h$, it follows that $P_h\bz=\b0$ if and only if $J^\T Q_h\bz=\b0$.  Notice that, by Theorem \ref{th:nohpi}, if the $h$-pseudoinverse exists then $S$ and $S^{\perp h}$ have trivial intersection.

\bigskip
\noindent{\bf Example~3.}  A second application to robot control problems is to find all possible joint velocities for an under-actuated robot arm, in a given non-singular configuration, so that a point on its end-effector moves in a given direction.  Suppose the arm has $m<6$ actuated joints with joint variables $\th=(\th_1,\dots,\th_m)\in M$ where $M$ is the joint space of the arm. Let $f:M\to SE(3)$ denote its kinematic mapping. Given a point $\bx$ on its end-effector,  the motion of $\bx$ is described by $f_\bx:M\to\bbR^3$ where $f_\bx(\th)=f(\th).\bx$.    In order for the point $\bx$ to move parallel to the direction $\bq$ at the configuration $\th$, we require to find $\dot{\th}\in T_\th M\cong \bbR^m$ and $\gl\in\bbR$ so that $Df_\bx(\th)\dot{\th}=\gl\bq$. 

To simplify things, one may recalibrate the joint variables so that $f(\th)$ is the identity in $SE(3)$. Then the derivative is a map $Df:\bbR^m\to\Lse(3)$. We can then write $Df(\th)$ as a Jacobian $J$ in which the columns are the $m$ joint twists. Since the action of $SE(3)$ on $\bbR^3$ is affine, the derivative $Df_\bx$ is just the composition $ev_\bx\circ J$ where $ev_\bx:\Lse(3)\to\bbR^3$ describes the action of the Lie algebra on space: $ev_\bx(\bgw,\bv)=\bgw\x\bx+\bv$. Denote by $L\subset\bbR^3$ the line $\{\gl\bq\,:\,\gl\in\bbR\}$. Thus, we require to solve the condition:
\begin{equation}
\label{e:ptmotion}
ev_\bx(J\dot{\th})\in L.
\end{equation}
Now $ev_\bx$ is surjective, so the inverse image ${ev_\bx}^{-1}(L)$ must be a subspace of the same codimension in $\Lse(3)$ and hence is a 4-system~$S$.  The problem therefore reduces to finding $\dot{\th}\in\bbR^m$ so that $J\dot{\th}\in S$. Assuming the existence of a pseudoinverse $J^{+h}$, we determine the solution set to be $\dot{\th}\in J^{+h}(S)$. Further, the 4-system $S$ can be identified as the Lie algebra of the cylinder subgroup consisting of screw motions (including rotation) about the line of motion $\bx + t\bq$ together with  translation in the direction of the velocity vector $\bv$.  This has Gibson--Hunt type II$A$, the 4-system being reciprocal to the corresponding 2-system in~\Cref{tab:2systems}. 

\subsection{Optimality}
Moore--Penrose pseudoinverses have the property that they determine optimal solutions to equations, in the sense that attempting to solve the linear system $A\bx=\bb$, the vector $\bv=A^+\bb$ (where $A$ is $n\x m$) minimises the least-squares error~\cite{penrose}:
\begin{equation}
\label{e:mpmin}
\forall\,\bx\in\bbR^m,\quad\|A\bv-\bb\|\leq\|A\bx-\bb\|,
\end{equation}
where the Euclidean norm is used in $\bbR^n$. In the same way,  hyperbolic pseudoinverses `optimise' an appropriate function, though the indefiniteness of the quadratic form $Q_h$ means that the approximation to the solution is not necessarily minimised. More precisely, the $h$-pseudoinverse gives a solution to a linear system of equations that is a stationary point of an appropriate cost function $\Phi_h$.
\begin{theorem}
\label{t:psopt}
Let $J$ be a $6\x m$ Jacobian matrix, $\bs\in \Lse(3)$ a twist and suppose that $\bv=J^{+h}\bs$. Define 
\begin{equation}
\label{e:hobjfn}
\Phi_h:\bbR^m\to\bbR,\quad \Phi_h(\bx)=(\bs-J\bx)^\T Q_h(\bs-J\bx).
\end{equation}
Then for each $i=1,\dots,m$,
\begin{equation}
\label{e:psopt}
\left.\frac{\d}{\d x_i}\Phi_h(\bx)\right|_{\bx=\bv}=0.
\end{equation}

\end{theorem}

\begin{proof}
In general the partial derivatives can be written in the matrix form,
\[
\frac{\d}{\d \bx}\Phi_h(\bx)=\begin{pmatrix}
\frac{\d}{\d x_1}\Phi_h(\bx)\\
\vdots\\
\frac{\d}{\d x_6}\Phi_h(\bx)
\end{pmatrix}= -2J^\T Q_h(\bs-J\bx).
\]
Setting $\bx=\bv=J^{+h}\bs$ gives,
\[
2\big((J^TQ_hJ)J^{+h}-J^TQ_h\big)\bs=\b0\]
since $J^{+h}=(J^TQ_hJ)^{-1}J^TQ_h$.
\end{proof}

\bigskip
\noindent{\bf Example~4.} With $J$ as in Example~1, and $\bs=\left(0, \frac35, \frac45, 1, -\frac45, \frac35\right)^\T$, we obtain for $h\neq0$,
\[
\bv= J^{+h}\bs=\frac1{160h}\begin{pmatrix}-48h\\-48h^2+160h-45\\48h^2-32h-15\end{pmatrix}
\]
and the error is given by
\[
\Phi_h=-\frac{9(80h^2+64h-25)}{800h}
\]
Note that there are two values of $h$ where the error $\Phi_h$ vanishes.  While one of these may appear a sensible choice therefore, in a practical control problem the twist $\bx$ will vary so that solutions meeting this requirement are unlikely to exist for a given $h$ and all $\bx$ of interest. 

There are several other methods used in practice, see \cite{Buss+Kim}  again.  In particular the damped least squares method.  A geometric version of this method can be derived simply from the optimality conditions considered above.  The cost function to consider is now,
\[\Psi(\bx)=(\bs-J\bx)^\T Q_h(\bs-J\bx)+\bx^T\Lambda\bx\]
where $\Lambda$ is a symmetric matrix; if $\Lambda$ is the identity matrix then $\bx^T\Lambda\bx$ is a sum of squares.   In this case the partial derivatives can be written as,
\[
\frac{\d}{\d \bx}\Psi(\bx)= -2J^\T Q_h(\bs-J\bx)+2\Lambda\bx.
\]
Setting this to zero and solving for $\bx$ gives,
\[\bx=(J^TQ_hJ+\Lambda)^{-1}J^TQ_h\bs.\]
This gives a modified projection operator, $J(J^TQ_hJ+\Lambda)^{-1}J^TQ_h$, which might be used in cases where the pseudo-inverse $J^{+h}$ doesn't exist.

However, these are not the most general projection operators that could be used.  Suppose that $S$ and $Z$ are complementary subspaces of twists, so that $S\oplus Z=\Lse(3)$.  Then it is always possible to find a projection operator onto $S$ which has $Z$ as its null space.  In order to construct this, consider the dual, $\Lse^*(3)$, of the Lie algebra of twists.  Elements of the dual  space are linear functionals on  twists,  called {\em co-twists} or {\em wrenches}.  The dual space to $Z$ is the space of wrenches that annul all the twists in $Z$, that is:
\[ 
Z^*=\{{\cal W}\in\Lse^*(3)\,:\,{\cal W}^{\T}\bz=0,\,\forall\, \bz\in Z\}.
\]
Given a basis for $Z^*$, form the matrix $K$ whose columns are the basis vectors (written with respect to the dual Pl\"{u}cker coordinates).  The required projection operator is then given by:
\begin{equation}
\label{e:proj}
P=J(K^{\T}J)^{-1}K^{\T}.
\end{equation}
Clearly this construction can always be used, in particular when no pseudo-inverse exists.  However, there remains the problem of which of the many possible projection operators to choose in a given setting.

\subsection{Shared and Hybrid Control}
In shared control and hybrid control the task of controlling the end-effector of a robot is split in two.  In hybrid control there is a force/moment control task and a simultaneous position control task.  For example, consider using a robot to write using a pen.  The robot must maintain a constant force normal to the paper but also control the position of the pen tip on the surface of the paper.  In shared control, the robot has two positional tasks to satisfy, one task is exclusively controlled by the robot, say limiting the end-effector to move on a particular surface.   For the other task, the desired motion on the surface is generated by a human operator.  Both these types of control methods rely on projecting the error signal given by the robot's sensors into a particular twist or wrench subspace.    Note that wrenches, elements of $\Lse^*(3)$, can be used to represent force/torque vectors, see \cite[Chap. 12]{selig}.

For shared control, assume that the task is split into a pair of $h$-reciprocal spaces of twists, $S_1$ and $S_2$,  where $S_1$ is the space of motions to be controlled by the robot and $S_2$ is the space of motions to be controlled by the human operator.  If the sensors reveal a positional error $\bs_e$, then this must be projected to $S_1$ for the control system to deal with.  The part of the error in $S_2$ is left for the operator to manage.  As the two subspaces are $h$-reciprocal, $S_2=S_1^{\perp h}$, the projection operator $P_h=JJ^{+h}$ introduced above can be used, where the columns of $J$ are a basis for $S_1$.  Dually, the same construction can be used to give a projection operator for $S_2$ by using a matrix of basis elements for this space.  If the two subspaces are not $h$-reciprocal but only complementary, then the projection operator defined in (\ref{e:proj}) can be used.

Hybrid control was introduced in the early 1980s, however the original formulation was flawed because the projection operators used were not invariant under a change of  coordinates~\cite{lipkin}.    There are two spaces to consider, a space of twists $S$ containing the possible motion twist that must be controlled and a space of wrenches $S^*\subset \Lse^*(3)$  which contains the force and torque of the robot's end-effector which is to be controlled.   

These spaces must be dual to each other, i.e. ${\cal W}^T\bs=0$ for all ${\cal W}\in S^*$ and all $\bs\in S$.   This ensures that the two tasks do not interfere with each other: the motions to be controlled do no work on the wrenches to be imposed by the robot.  The projection operator $P_h=JJ^{+h}$ can be used to project a positional error onto  $S$. Force/torque error can be projected to $S^*$ with a similarly defined projection operator.  Let $W$ be the matrix whose columns are basis vectors for $S^*$.  We can define a pseudoinverse for this matrix as above,
\[
W^{+h}=(W^{\T}Q_h^{-1}W)^{-1}W^{\T}Q_h^{-1}.
\]
It can be verified that the invariant quadrics for the co-adjoint action of $SE(3)$ on $\Lse^*(3)$ are the inverses of those for the adjoint action:
\[
Q_h^{-1}=\begin{pmatrix}-2hI_3&I_3\\I_3&0\end{pmatrix}^{-1}=
\begin{pmatrix}0&I_3\\I_3&2hI_3\end{pmatrix}.
\]
The projection operator is then $P^*_h=WW^{+h}$.  Of course, other pairs of projection operators exist, as in equation (\ref{e:proj}) above.

\section{Conclusion}
There are other applications of pseudoinverses in robotics, for example to the computation of stiffness matrices for parallel manipulators.
The closely related concept of singular values and singular value and polar decomposition~\cite{bolsh,lins} is also much used in robotics and suffers from the same problems as the Moore--Penrose pseudoinverse.  That is, the standard singular value decomposition of a matrix of twists, or wrenches, is not well-defined under rigid changes of coordinates. We suggest that the approach used in this paper may offer advantages in all these problems.

\bibliographystyle{siamplain}
\bibliography{eucl_psi_references_simax-final}

\end{document}